\documentclass[12pt]{amsart}
\usepackage{amsmath,latexsym,amscd,amsbsy,amssymb,amsfonts,amsthm,fleqn,leqno,
euscript, graphicx, texdraw, pb-diagram,tikz}
\usepackage[matrix,arrow,curve]{xy}

\numberwithin{equation}{section}
\newtheorem{thm}{Theorem}[section]
\newtheorem{pro}[thm]{Proposition}

\newtheorem{cor}[thm]{Corollary}

\theoremstyle{definition}
\newtheorem{dfn}[thm]{Definition}

\newtheorem{que}[thm]{Question}

\theoremstyle{remark}

\numberwithin{equation}{section}

\begin{document}

\title[On homologically locally connected spaces]
{On homologically locally connected spaces}

\author{A. Koyama}
\address{Faculty of Science and Engineering,
Waseda University, Ohkubo 3-4-1, Shinjuku,Tokyo,
169-8555, Japan} \email{akoyama@waseda.jp}

\author{V. Valov}
\address{Department of Computer Science and Mathematics,
Nipissing University, 100 College Drive, P.O. Box 5002, North Bay,
ON, P1B 8L7, Canada} \email{veskov@nipissingu.ca}

\thanks{The author was partially supported by NSERC
Grant 261914-13.}


\begin{abstract}
We provide some properties and characterizations of homologically $UV^n$-maps
and $lc^n_G$-spaces.
We show that there is a parallel between  recently introduced by Cauty \cite{ca} algebraic $ANR$'s
and homologically $lc^n_G$-metric spaces,
and this parallel is similar to the parallel between ordinary $ANR$'s and $LC^n$-metric spaces.
We also show that
there is a similarity between the properties of $LC^n$-spaces and $lc^n_G$-spaces.
Some open questions are raised.
\end{abstract}
\maketitle\markboth{}{Homologically locally connected spaces}




\section{Introduction}

All spaces are assumed to be paracompact and all single-valued maps are continuous.
Everywhere below singular homology $H_n(X;G)$, reduced in dimension 0,
with a coefficient group $G$ is considered.
By default,
if not explicitly stated otherwise,
$G$ is a commutative ring with a unit $\rm e$.
The following homology counterpart of the well known notion of a $UV^n$-set
was introduced in \cite{vv1}:
A closed set $A\subset X$ is said to be {\em homologically $UV^n(G)$}
if every neighborhood $U$ of $A$ in $X$ contains another neighborhood $V$
such that the inclusion $V\hookrightarrow U$ induces trivial homomorphisms
$H_k(V;G)\to H_k(U;G)$
(notation $A\stackrel{H_k}{\hookrightarrow}X$) for all $k\leq n$.
Obviously,
every $UV^n$-subset of $X$ is $UV^n(G)$ for any $G$
(below we call the $UV^n$-sets {\em homotopically $UV^n$}
in order to distinguish them from homologically $UV^n(G)$-sets).
It can be shown that if $A$ is homologically $UV^n(\mathbb Z)$,
where $\mathbb Z$ is the group of the integers,
then it is $UV^n(G)$ for any $G$ (see, for example  \cite[Proposition 4.8]{vv}).
Moreover,
following the proof of Proposition 7.1.3 from \cite{sa},
one can show that $A$ is homologically $UV^n(G)$ in a given metric $ANR$-space $X$
if and only if it is homologically $UV^n(G)$ in any metric $ANR$-space
that contains $A$ as a closed set.
We also say that a surjective map $f:X\to Y$ is {\em homologically $UV^n(G)$}
if all fibres $f^{-1}(y)$ of $f$ are homologically $UV^n(G)$-subsets of $X$.
In particular,
$X$ is called {\em homologically locally connected in dimension $n$ with respect to $G$}
if the identity map on $X$ is an $UV^n(G)$-map,
notation $X$ is $lc^n_G$.

In this paper we provide some properties and characterizations of homologically $UV^n$-maps
and $lc^n_G$-spaces.
We show that there is a parallel between  recently introduced by Cauty \cite{ca}
algebraic $ANR$'s and homologically $lc^n_G$-metric spaces,
and this parallel is similar to the parallel between ordinary $ANR$'s and $LC^n$-metric spaces.
We also show that there is an analogy
between the properties of $LC^n$-spaces and $lc^n_G$-spaces.

More precisely,
Section 2 contains some definitions and properties of algebraic $ANR$'s
about the existence and extensions of homotopies between close chain morphisms.
Section 3 is devoted to closed homotopically $UV^n$-surjections.
We show that the well known properties of closed $UV^n$-maps
(concerning extensions of partial realizations and approximate lifting of maps)
have chain morphisms' analogues for homologically $UV^n$-maps.
We also provide some Dugundji type extensions and approximate extensions of chain morphisms.
Obviously, all results established in Section 3 are valid and for $lc^n_G$-spaces.
In Section 4 we characterize $lc^n_G$-spaces in terms of existence of homotopies
between any close chain morphisms (Proposition 4.1).
The following two questions are discussed in that section:
\begin{que}
Is any metric $n$-dimensional $lc^n_G$-space an absolute neighborhood algebraic retract?
\end{que}

\begin{que}
Let $f:Y\to X$ be a closed $UV^n(G)$-surjection between metric spaces.
Is it true that $X$ is an  $lc^n_G$-space?
\end{que}

In case of $LC^n$-spaces and $UV^n$-maps the above two questions have positive answers.
We still do not know the answer to these questions,
but we introduce approximate versions of algebraic $ANR$-spaces and $lc^n_G$-spaces
and show that any space satisfying the hypotheses of Question 1.1 (resp., Question 1.2)
is an approximate $ANR$ (resp., approximate $lc^n_G$).

The main technical tool are chain morphisms between chain complexes.
By a chain complex $C=\{C_k\}_{k\geq 0}$
we mean a sequence of $G$-modules $C_k$ and
boundary homomorphisms $\partial_k:C_k\to C_{k-1}$ such that
all compositions $\partial_{k}\circ\partial_{k+1}$ are trivial.
A chain morphism $\varphi :C\to C'$ between two chain complexes
is a sequence of homomorphisms $\varphi_k:C_k\to C_k'$
such that $\varphi_k\circ\partial_{k+1}=\partial_{k+1}'\circ\varphi_{k+1}$,
and $\varphi_0$ commutes with the augmentations in $C$ and $C'$.
We consider two types of chain complexes,
oriented chain complexes $C(K)=\{C_k(K;G)\}_{k\geq 0}$,
where $K$ is a simplicial complex, and
singular chain complexes $S(X;G)=\{S_k(X;G)\}_{k\geq 0}$,
where $X$ is a given topological space and
$S_k(X;G)$ is the group of all singular $k$-chains with coefficients from $G$.
If $\sigma:\triangle^k\to X$ is a singular $k$-simplex ($\triangle^k$ is the standard $k$-simplex),
we denote by $|\sigma|$ the {\em carrier} $\sigma(\triangle^k)$ of $\sigma$.
Similarly, we put $|c|=\bigcup_i|\sigma_i|$ for any chain $c\in S_k(X;G)$,
where $c=\sum_ig_i\sigma_i$ is the irreducible representation of $c$.
We agree that $|c|=\varnothing$ if $c=0$.
For any open cover $\mathcal U$ of $X$ let $S(X,\mathcal U;G)$ stand
for the subgroup of $S(X;G)$ generated by all singular simplexes $\sigma$
with $|\sigma|\subset U$ for some $U\in\mathcal U$.

By a sub-complex of $S(X;G)$
we mean any sub-family $\mathrm S\subset S(X;G)$
such that $\partial c\in\mathrm S$ for any chain $c\in\mathrm{S}$.
For example, $S(X,\mathcal U;G)$ and $S(A;G)$ are sub-complexes of $S(X;G)$
for any open cover $\mathcal U$ of $X$ and any subset $A\subset X$.
We also consider the $n$-dimensional sub-complex
$S^{(n)}(X;G)=\{S_k(X;G)\}_{k\leq n}$ of $S(X;G)$.
A face $\tau$ of a singular simplex $\sigma:\triangle^k\to X$ from $S(X;G)$ is
the restriction of the map $\sigma$ onto a face of the standard simplex $\triangle^k$.
In particular, a vertex of $\sigma$ is the singular $0$-simplex $f_v:v\to X$
with $v$ being a vertex of $\triangle^k$.
If $Y$ is a space and $\mathcal U$ is an open cover of $Y$,
we say two chain morphisms
$\varphi,\psi:S(X;G)\to S(Y;G)$ (resp., $\varphi,\psi:C(K;G)\to S(Y;G)$) are {\em $\mathcal U$-close}
provided for each simplex $\sigma\in S(Y;G)$ (resp., $\sigma\in K$)
there exists $U_\sigma\in\mathcal U$ with
$|\varphi(\tau)|\cup |\psi(\tau)|\subset U_\sigma$ for all faces $\tau$ of $\sigma$.
If $\varphi$ is $\mathcal U$-close to the trivial morphism,
then $\varphi$ is said to be {\em $\mathcal U$-small}.
We also say that $\varphi$ is {\em correct} provided $\varphi(\sigma)$ (resp., $\varphi(v)$)
is a singular $0$-simplex in $S(Y;G)$ for each $0$-simplex $\sigma\in S(X;G)$
(resp., for each vertex $v$ of $K$).




\section{Algebraic $ANR$'s}

Let $K$ be a simplicial complex,
$X$ be a space and $\mathcal U$ an open cover of $X$.
According to \cite{ca},
a chain morphism $\varphi:C(L;G)\to S(X;G)$ is a
{\em partial algebraic realization of $C(K;G)$} in $\mathcal U$
provided $L$ is a sub-complex of $K$ containing the vertex set $K^{(0)}$ of $K$
and for every simplex $\sigma\in K$ there exists $U_\sigma\in\mathcal U$
such that $|\varphi(\tau)|\subset U_\sigma$ for all faces $\tau$ of $\sigma$ with $\tau\in L$.
If $L=K$, then $\varphi$ is called a {\em full algebraic realization of $C(K;G)$ in $\mathcal U$}.
Obviously, $\varphi:C(K;G)\to S(X;G)$ is a full algebraic realization of $C(K;G)$ in $\mathcal U$
if and only if $\varphi$ is $\mathcal U$-small.

Cauty \cite{ca} introduced an important class of metric spaces more general than metric $ANR$'s.
\begin{dfn}\cite{ca}
A metric space $X$ is said to be an {\em absolute neighborhood algebraic $G$-retract}
(briefly, {\em algebraic $ANR_G$})
if for every open cover $\mathcal U$ of $X$
there is another open cover $\mathcal V$ refining $\mathcal U$
such that, for any simplicial  complex $K$,
any correct partial algebraic realization of  $C(K;G)$ in $\mathcal V$
extends to a full algebraic realization of $C(K;G)$ in $\mathcal U$.
Any acyclic algebraic $ANR_G$ is called an {\em algebraic absolute $G$-retract}
(br., algebraic $AR_G$).
\end{dfn}

Here are some properties of algebraic $ANR_G$'s.
\begin{dfn}\cite{ca}
A closed subset $A$ of a space $X$ is said to be a {\em neighborhood  algebraic $G$-retract} of $X$
if there are a neighborhood $U$ of $A$ in $X$, an open cover $\mathcal U$ of $U$ and
a chain morphism $\mu:S(U,\mathcal U;G)\to S(A;G)$ such that:
\begin{itemize}
\item[(i)] $\mu(c)=c$ for all $c\in S(A;G)\cap S(U,\mathcal U;G)$;
\item[(ii)] For every $x\in A$ and its neighborhood $V_x$ in $A$
there exists a neighborhood $W_x\subset U$
with $\mu\big(S(W_x;G)\cap S(U,\mathcal U;G)\big)\subset S(V_x;G)$.
\end{itemize}
If $U=X$ and $\mathcal U=\{X\}$, $A$ is called an {\em algebraic $G$-retract of $X$}.
\end{dfn}

\begin{thm}\cite{ca}
A metric space $X$ is an algebraic $ANR_G$ if and only if
$X$ is a neighborhood  algebraic $G$-retract of every metric space containing $X$ as a closed set.
\end{thm}

Next result is an analogue of the $ANR$'s properties concerning close maps and small homotopies.
Recall that if $\varphi, \phi: S(Y;G)\to S(X;G)$ are two chain morphisms,
then $D:S(Y;G)\to S(X;G)$ is a chain homotopy between $\varphi$ and $\psi$
provided there exists a sequence of homomorphisms
$D_k:S_k(Y;G)\to S_{k+1}(X;G)$, $k\geq 0$, such that
$\partial D_0(\sigma)=\varphi(\sigma)-\psi(\sigma)$
for any singular $0$-simplex $\sigma\in S_0(Y;G)$ and
$\partial D_k(\sigma)=\varphi(\sigma)-\psi(\sigma)-D_{k-1}(\partial\sigma)$
provided $\sigma\in S_k(Y;G)$ is a singular $k$-simplex with $k\geq 1$.
We say that $D$ is {\em $\mathcal U$-small},
where $\mathcal U$ is an open cover of $X$
if for any $\sigma\in S(Y;G)$ there exists $U_\sigma\in\mathcal U$
such that $|D(\tau)|\cup |\varphi(v)|\cup|\psi(v)|\subset U_\sigma$
for all singular simplexes $\tau\in S(Y;G)$,
which are faces $\tau$ of $\sigma$, and all vertexes $v$ of $\sigma$.

\begin{pro}
If $X$ is metric algebraic $ANR_G$,
then for any open cover $\mathcal U$ of $X$
there is an open cover $\mathcal V$ of $X$ refining $\mathcal U$ such that,
for any two correct $\mathcal V$-close chain morphisms $\varphi, \phi: S(Y;G)\to S(X;G)$,
where $Y$ is an arbitrary space, and any $\mathcal V$-small chain homotopy $D:S(A;G)\to S(X;G)$
between $\varphi|S(A;G)$ and $\phi|S(A;G)$ with $A$ being a closed set in $Y$,
there exists a $\mathcal U$-small chain homotopy $\Phi:S(Y;G)\to S(X;G)$
between $\varphi$ and $\phi$ such that $\Phi(c)=D(c)$ for all $c\in S(A;G)$.
\end{pro}

\begin{proof}
Suppose $X$ is an algebraic $ANR_G$ and $\mathcal U$ is an open cover of $X$.
Embed $X$ as closed subset of a normed space $E$ and let
$\mu:S(W,\mathcal W;G)\to S(X;G)$ be an algebraic neighborhood $G$-retraction,
where $W$ is a neighborhood of $X$ in $E$ and $\mathcal W$ an open cover of $W$.
For every $x\in X$ choose $U_x\in\mathcal U$ and $W_x\in\mathcal W$ containing $x$.
We can assume that $U_x$ and $W_x$ satisfy the inclusion
$\mu(S(W_x;G))\subset S(U_x;G)$ for every $x\in X$.
Let $V_x$ be a convex open subset of
$W$ such that $V_x\cap X\subset U_x$ and $V_x\subset W_x$, $x\in X$.
Denote $T=\bigcup_{x\in X}V_x$, $\widetilde{\mathcal V}_1=\{V_x:x\in X\}$
and take an open cover $\widetilde{\mathcal V}$ of $T$ consisting of convex sets
such that
$\widetilde{\mathcal V}$ is a star refinement of $\widetilde{\mathcal V}_1$.
Let $\mathcal V=\{\widetilde V\cap X:\widetilde V\in\widetilde{\mathcal V}\}$.

Now, let $\varphi,\phi:S(Y;G)\to S(X;G)$ be two correct $\mathcal V$-close chain morphisms
and $D:S(A;G)\to S(X;G)$ be a $\mathcal V$-small chain homotopy
between $\varphi|S(A;G)$ and $\phi|S(A;G)$,
where $Y$ is a space and $A\subset Y$ is closed.

Everywhere below we say that a set $B\subset T$ is $\widetilde{\mathcal V}_1$-small
if $B$ is contained in some element of $\widetilde{\mathcal V}_1$.

 \smallskip
\textit{Claim $1$.
For every singular simplex $\sigma\in S(Y;G)$ there is a non-empty convex
$\widetilde{\mathcal V}_1$-small set $\Lambda_\sigma\subset T$ such that:
\begin{itemize}
\item $\Lambda_\tau\subset\Lambda_\sigma$ and
$|\varphi(\tau)|\cup |\phi(\tau)|\subset\Lambda_\sigma$ for all faces $\tau$ of $\sigma$;
\item If $\sigma\in S(A;G)$, then $\Lambda_\sigma$ contains also $|D(\tau)|$,
$\tau$ is a face of $\sigma$.
\end{itemize}}

Indeed,
since $\varphi$ and $\phi$ are $\mathcal V$-close,
for every singular simplex $\sigma\in S(Y;G)$
there exists $\widetilde V_\sigma\in\widetilde{\mathcal V}$ such that
$|\varphi(\tau)|\cup |\phi(\tau)|\subset\widetilde V_\sigma$ for all singular simplexes $\tau\in S(Y;G)$
which are faces of $\sigma$.
In particular,
$\widetilde V_\sigma$ contains the non-empty sets $|\varphi(v)|\cup |\phi(v)|$,
$v$ is a vertex of $\sigma$
(recall that $|\varphi(v)|$ and $|\phi(v)|$ are both non-empty because $\varphi$ and $\phi$ are correct).
Similarly, using that $D$ is $\mathcal V$-small,
for any $\sigma\in S(A;G)$
we can find $\widetilde V_\sigma^D\in\widetilde{\mathcal V}$ containing all
$|D(\tau)|\cup |\varphi(v)|\cup|\phi(v)|$,
where $\tau$ is a face of $\sigma$ and $v$ is a vertex of $\sigma$.
On the other hand,
$|\varphi(v)|\cup |\phi(v)|\subset\widetilde V_\sigma$.
Therefore, $\widetilde V_\sigma\cap\widetilde V_\sigma^D\neq\varnothing$ for any singular simplex
$\sigma\in S(A;G)$. 
Next, for any singular simplex $\sigma\in S(Y;G)$ let
$$
\Gamma_\sigma=\bigcap\{\widetilde V_s\in\widetilde{\mathcal V}:\sigma\hbox{~}
\mbox{is a face of a singular simplex}\hbox{~}s\in S(Y;G)\}
$$
and
$$
\Gamma_\sigma^D=\bigcap\{\widetilde V_s^D\in\widetilde{\mathcal V}:s\in S(A;G)\hbox{~}
\mbox{is a simplex and}\hbox{~}\sigma\hbox{~}\mbox{is a face of}\hbox{~}s\}.
$$
Obviously,
$\Gamma_\sigma\neq\varnothing$ for any $\sigma\in S(Y;G)$,
and $\Gamma_\sigma^D\neq\varnothing$ provided $\sigma\in S(A;G)$.
We assume that $\Gamma_\sigma^D=\varnothing$ if $\sigma\not\in S(A;G)$.
Consider the sets
$\Omega_\sigma=\Gamma_\sigma\bigcup\{\Gamma_\tau^D:\tau\hbox{~}
\mbox{is a face of}\hbox{~}\sigma\}$,
$\sigma\in S(Y;G)$.
Observe that if a singular simplex $\tau\in S(A;G)$ is a face of some $\sigma\in S(Y;G)$,
then $\Gamma_\sigma\cap\Gamma_\tau^D\neq\varnothing$
because this set contains all $|\varphi(v)|\cup|\phi(v)|$ with $v$ being a vertex of $\tau$.
Moreover, each of the sets $\Gamma_\sigma$ and $\Gamma_\tau^D$ are
$\widetilde{\mathcal V}$-small sets. So,
$\Omega_\sigma\subset\rm{St}(\Gamma_\sigma,\widetilde{\mathcal V})$,
$\sigma\in S(Y;G)$,
where $\rm{St}(\Gamma_\sigma,\widetilde{\mathcal V})$ denotes
the star of the set $\Gamma_\sigma$ with respect to $\widetilde{\mathcal V}$.
Thus,
$\Omega_\sigma$ is a $\widetilde{\mathcal V}_1$-small subset of $T$
(recall that $\widetilde{\mathcal V}$ is a star refinement of $\widetilde{\mathcal V}_1$).
This implies that the convex hull $\Lambda_\sigma$ of $\Omega_\sigma$ is a convex non-empty
$\widetilde{\mathcal V}_1$-small subset of $T$.
It follows from the definitions of $\Gamma_\sigma$ and $\Gamma_\sigma^D$
that $\Omega_\sigma$ contains $\Omega_\tau$ for all faces $\tau$ of $\sigma$,
so $\Lambda_\tau\subset\Lambda_\sigma$.
Because for any singular simplex $\sigma\in S(Y;G)$ (resp., $\sigma\in S(A;G)$) and
a face $\tau$ of $\sigma$ the set  $\Gamma_\sigma$ contains $|\varphi(\tau)|\cup |\phi(\tau)|$
(resp., $\Gamma_\sigma^D$ contains $|D(\tau)|$),
the sets $\Lambda_\sigma$ have the same property.
This completes the proof of Claim 1.

We are going to construct first a chain homotopy
$\widetilde D:S(Y;G)\to S(T,\widetilde{\mathcal V}_1;G)$ between $\varphi$ and $\phi$
such that $\widetilde D(c)=D(c)$ for all $c\in S(A;G)$.
To this end, observe that
both $\varphi(\sigma)$ and $\phi(\sigma)$ are singular $0$-simplexes in $S(X;G)$
for any singular $0$-simplex $\sigma\in S(Y;G)$.
Then $\varphi(\sigma)-\phi(\sigma)$ is a $0$-cycle in $S_0(\Lambda_\sigma;G)$.
So, there is a singular $1$-chain $c^1_\sigma\in S_1(\Lambda_\sigma;G)$
with $\partial c^1_\sigma=\varphi(\sigma)-\phi(\sigma)$
(recall that $H_0(\Lambda_\sigma;G)=0$ because $\Lambda_\sigma$ is a convex set).
We define $D_0'(\sigma)=c^1_\sigma$ if $\sigma\not\in S_0(A;G)$ and
$D_0'(\sigma)=D_0(\sigma)$ if $\sigma\in S_0(A;G)$.
Thus,
$|D_0'(\sigma)|\subset\Lambda_\sigma$ for all singular $0$-simplexes $\sigma$ of $S(Y;G)$.
So, we can extend $D_0'$ linearly to a homomorphism
$\widetilde D_0:S_0(Y;G)\to S_1(T,\widetilde{\mathcal V}_1;G)$.

Assume that homomorphisms $\widetilde D_k:S_k(Y;G)\to S_{k+1}(T,\widetilde{\mathcal V}_1;G)$
were defined for all $k\leq m$ such that:
\begin{itemize}
\item[(1)] $\partial\widetilde D_k(c^k)=\varphi(c^k)-\phi(c^k)-\widetilde D_{k-1}(\partial c^k)$
for all $c^k\in S_k(Y;G)$;
\item[(2)] $\widetilde D_k(c^k)=D_k(c^k)$ for all $c^k\in S(A;G)$;
\item[(3)] $|\widetilde D_i(\tau)|\subset\Lambda_\sigma$
for any singular $k$-simplex $\sigma\in S_{k}(Y;G)$ and
any singular $i$-simplex $\tau$, which is a face of $\sigma$.
\end{itemize}

To define $\widetilde D_{m+1}$,
let $\sigma$ be a singular $(m+1)$-simplex in $S_{m+1}(Y;G)$.
Then $|\varphi(\sigma)|\cup |\phi(\sigma)|\subset\Lambda_\sigma$ and,
according to $(3)$,
$\Lambda_\sigma$ contains also $|\widetilde D_{m}(\partial\sigma)|$.
Therefore,
the chain $\gamma_\sigma=\varphi(\sigma)-\phi(\sigma)-\widetilde D_{m}(\partial\sigma)$
belongs to $S_{m+1}(\Lambda_\sigma;G)$.
It is easily seen that
$\gamma_\sigma$ is a cycle,
and since $\Lambda_\sigma$ is convex,
there is a chain $c^{m+2}_\sigma\in S_{m+2}(\Lambda_\sigma;G)$
with $\partial c^{m+2}_\sigma=\gamma_\sigma$.
We define $D'_{m+1}(\sigma)=c^{m+2}_\sigma$  if $\sigma\not\in S_{m+1}(A;G)$ and
$D'_{m+1}(\sigma)=D(\sigma)$ if $\sigma\in S_{m+1}(A;G)$.
According to the properties of $\Lambda_\sigma$,
we always have
$|D'_{m+1}(\sigma)|\subset\Lambda_\sigma$, $\sigma\in S_{m+1}(Y;G)$.
Therefore, we can extend $D'_{m+1}$ to a homomorphism
$\widetilde D_{m+1}:S_{m+1}(Y;G)\to S_{m+2}(T,\widetilde{\mathcal V}_1;G)$.
Obviously, $\widetilde D_{m+1}$ satisfies conditions $(1)-(3)$.
This completes the construction of $\widetilde D$.

Finally,
note that $S(T,\widetilde{\mathcal V}_1;G)\subset S(W,\mathcal W;G)$
because $T\subset W$ and $\widetilde{\mathcal V}_1$ refines $\mathcal W$.
Hence, for any $k\geq 0$ the homomorphism $\Phi_k:S_k(Y;G)\to S_{k+1}(X;G)$,
$\Phi_k=\mu_{k+1}\circ\widetilde D_k$, is well defined.
According to $(1)$
we have $\partial\Phi_k(c^k)=\varphi_k(c^k)-\phi_k(c^k)-\Phi_{k-1}(\partial c^k)$
for all $c^k\in S_k(Y;G)$.
Therefore,
$\Phi$ is a chain homotopy between $\varphi$ and $\phi$ and,
since $\mu(S(W_x;G))\subset S(U_x;G)$ for every $x\in X$, $\Phi$ is $\mathcal U$-small.
Because $D$ is $\mathcal V$-small,
$c\in S(A;G)$ implies $D(c)\in S(X;G)\cap S(T,\widetilde{\mathcal V}_1;G)$.
So, $\Phi(c)=D(c)$ for all $c\in S(A;G)$ (recall that $\mu$ is an algebraic retraction).
Thus, $\Phi$ extends $D$.
\end{proof}

If the set $A$ in Proposition 2.4 is empty, we obtain the following corollary.
\begin{cor}
Let $X$ be a metric algebraic $ANR_G$.
Then for every open cover $\mathcal U$ of $X$
there is an open cover $\mathcal V$ of $X$ refining $\mathcal U$
such that any two correct $\mathcal V$-close morphisms $\varphi,\phi:S(Y;G)\to S(X;G)$,
where $Y$ is a metric space, are $\mathcal U$-homotopic.
\end{cor}

Proposition 2.4 and Corollary 2.5 remain true
if the singular complex $S(Y;G)$ and the space $A\subset Y$ are replaced
by a simplicial complex $C(K;G)$ and a sub-complex $L$ of $K$, respectively.

\begin{pro}
If $X$ is metric algebraic $ANR_G$,
then for any open cover $\mathcal U$ of $X$
there is an open cover $\mathcal V$ of $X$ refining $\mathcal U$ such that,
for any two correct $\mathcal V$-close chain morphisms $\varphi, \phi: C(K;G)\to S(X;G)$,
where $K$ is a simplicial complex, and
any $\mathcal V$-small chain homotopy $D:C(L;G)\to S(X;G)$
between $\varphi|C(L;G)$ and $\phi|C(L;G)$ with $L$ being a sub-complex of $K$,
there exists a chain $\mathcal U$-small homotopy $\Phi:C(K;G)\to S(X;G)$
between $\varphi$ and $\phi$ extending $D$.
\end{pro}

\begin{cor}
Let $X$ be a metric algebraic $ANR_G$.
Then for every open cover $\mathcal U$ of $X$
there is an open cover $\mathcal V$ of $X$ refining $\mathcal U$
such that any two correct $\mathcal V$-close morphisms $\varphi,\phi:C(K;G)\to S(X;G)$,
where $K$ is a simplicial complex, are $\mathcal U$-homotopic.
\end{cor}


\section{Homologically $UV^n$-maps}
 In this section we establish some properties of closed $UV^n(G)$-maps.
We start this section with the following statement,
which an analogue of the corresponding result for homotopically $UV^n$-maps,
see for example \cite{du}.

\begin{pro}
Let $f:X\to Y$ be a closed homologically $UV^n(G)$ surjection between paracompact spaces.
Then for every open cover $\mathcal U$ of $Y$
there is an open cover $\mathcal V$ of $Y$ refining $\mathcal U$ such that
any correct partial algebraic realization of $C(K;G)$ in $f^{-1}(\mathcal V)$,
where $K$ is an $(n+1)$-dimensional simplicial complex,
extends to a full algebraic realization $\phi: C(K;G)\to S(X;G)$ of $C(K;G)$ in $f^{-1}(\mathcal U)$.
\end{pro}

\begin{proof}
Denote $\mathcal U$ by $\mathcal U_{n+1}$ and
for every $y\in Y$ fix $U_{n+1}(y)\in\mathcal U_{n+1}$ containing $y$.
We are going to construct by induction open covers $\mathcal U_k$  of $Y$ for each $k=n,n-1,..,0$
following the proof of \cite[Theorem 3.1]{du}.
Since $f$ is a closed map and each fiber $f^{-1}(y)$ is a homologically $UV^n(G)$-set in $X$,
for every $y\in Y$ there exists a neighborhood $V_{n+1}(y)$ of $y$ such that
$f^{-1}\big(V_{n+1}(y)\big)\stackrel{H_n}{\hookrightarrow}f^{-1}\big(U_{n+1}(y)\big)$.
Let $\mathcal U_n$ be an open star-refinement of the cover $\{V_{n+1}(y):y\in Y\}$.
If $\mathcal U_{k+1}$ is already defined,
we repeat the above construction to obtain for each $y\in Y$ its neighborhood $V_{k+1}(y)$
with $f^{-1}\big(V_{k+1}(y)\big)\stackrel{H_k}{\hookrightarrow}f^{-1}\big(U_{k+1}(y)\big)$,
and then take $\mathcal U_k$ to be an open star-refinement of the cover $\{V_{k+1}(y):y\in Y\}$.
We proceed recursively until construct $\mathcal U_0$.

Let show that $\mathcal V=\mathcal U_0$ is the required cover.
Suppose $K$ is an $(n+1)$-dimensional simplicial complex and
$\varphi:C(L;G)\to S(X;G)$ a correct partial algebraic realization of $C(K;G)$ in $f^{-1}(\mathcal U_0)$,
where $L$ is a sub-complex of $K$ containing all vertices of $K$.
For every $k\geq 1$ we are going to construct a homomorphism
$\phi_k:C_k(K;G)\to S_k(X;G)$ extending $\varphi_k:C_k(L;G)\to S_k(X;G)$
such that $\partial^X\circ\phi_k=\phi_{k-1}\circ\partial$ and satisfying the following condition $(4_k)$
(here, $\partial$ and $\partial^X$ are the boundary operators in $C(K;G)$ and $S(X;G)$, respectively).

\begin{enumerate}
\item[$(4_k)$] For every $k$-simplex $\sigma$ in $K$
there exist $U^k_\sigma\in\mathcal U_k$  such that
$|\phi_i(\tau)|\subset f^{-1}(U^k_\sigma)$ for any $i\leq k$ and any $i$-dimensional face $\tau$ of $\sigma$.
\end{enumerate}

To construct $\phi_1$,
let $\sigma=(v_0,v_1)$ be a $1$-simplex from $K$ and take $U_\sigma^0\in\mathcal U_0$
such that $f^{-1}(U^0_\sigma)$ contains $|\varphi(v_i)|$, $i=0,1$.
Moreover, if $\sigma\in L$ we can suppose that $f^{-1}(U^0_\sigma)$ also contains $|\varphi(\sigma)|$.
Since $\mathcal U_0$ is a star refinement of $\{V_{1}(y):y\in Y\}$,
$U_\sigma^0\subset f^{-1}(V_{1}(y_\sigma))$ for some $y_\sigma\in Y$.
So,
$\varphi(\partial\sigma)=\varphi(v_1)-\varphi(v_0)$ is a singular $0$-cycle in $f^{-1}(V_{1}(y_\sigma))$
(recall that each $\varphi(v_i)$, $i=1,2$, is
a singular $0$-simplex and $\epsilon(\varphi(v_1)-\varphi(v_0))=0$,
where $\epsilon: S_0(X;G)\to G$ is the augmentation of $S(X;G)$).
Since $f^{-1}\big(V_{1}(y_\sigma)\big)\stackrel{H_1}{\hookrightarrow}f^{-1}\big(U_{1}(y_\sigma)\big)$,
there exists a $1$-chain $c_{\sigma}^1\in f^{-1}\big(U_{1}(y_\sigma)\big)$
such that $\partial^X(c_{\sigma}^1)=\varphi(\partial\sigma)$.
We define $\phi_1'(\sigma)=c_\sigma^1$
if $\sigma\not\in L$ and $\phi'_1(\sigma)=\varphi(\sigma)$ if $\sigma\in L$,
and extend $\phi_1'$ linearly to a homomorphism $\phi_1:C_1(K;G)\to S_1(X;G)$.
Then condition $(4_1)$ is satisfied with $U^1_\sigma=U_{1}(y_\sigma)$
because $|\phi(v_0)|\cup |\phi(v_1)|\cup |\phi_1(\sigma)|\subset f^{-1}\big(U_{1}(y_\sigma)\big)$.

Suppose that the homomorphisms $\phi_i$ have been already constructed for some $k>1$ and
all $i\leq k$, and let $\sigma$ be a $(k+1)$-simplex of $K$.
Choose $U_0\in\mathcal U_0$ and $U_0^k\in\mathcal U_k$
such that $f^{-1}(U_0)$ contains all $|\varphi(v)|$ with $v\in\sigma^{(0)}$ and $U_0\subset U_0^k$.
In case $\sigma\in L$, we can suppose that $|\varphi(\sigma)|$ is also contained in $f^{-1}(U_0)$.
Then, by $(4_k)$, for every $k$-simplex $\tau$, which is a face of $\sigma$,
there exist $U^k_\tau\in\mathcal U_k$ such that
$|\phi_i(s)|\subset f^{-1}(U^k_\tau)$ for all $i$-dimensional faces $s$ of $\tau$, $i\leq k$.
In particular, $f^{-1}(U^k_\tau)$ contains $|\varphi(v)|$ for all vertexes $v$ of $\tau$.
Because $|\varphi(v)|\neq\varnothing$ for all $v\in K^{(0)}$ (recall that $\varphi$ is correct),
we have
$|\varphi(v)|\subset f^{-1}(U_0^k)\cap f^{-1}(U^k_\tau)\neq\varnothing$
for all faces $\tau$ of $\sigma$ and all $v\in\tau^{(0)}$.
Consequently,
$f^{-1}\big(\rm{St}(U_0^k,\mathcal U_k)\big)\neq\varnothing$ and
$|\phi_k(\partial(\sigma))|\subset f^{-1}\big(\rm{St}(U_0^k,\mathcal U_k)\big)$,
where $\rm{St}(U_0^k,\mathcal U_k)$ is the star of $U_0^k$ with respect to $\mathcal U_k$.
Since $\mathcal U_k$ is a star-refinement of $\{V_{k+1}(y):y\in Y\}$,
$|\phi_k(\partial(\sigma))|\subset f^{-1}(V_{k+1}(y_\sigma))$ for some $y_\sigma\in Y$.
Hence,
$\phi_k(\partial(\sigma))$ is a singular $k$-cycle in $f^{-1}(V_{k+1}(y_\sigma))$.
Finally,
since $f^{-1}\big(V_{k+1}(y_\sigma)\big)\stackrel{H_k}{\hookrightarrow}f^{-1}\big(U_{k+1}(y_\sigma)\big)$,
there exists a $(k+1)$-chain  $c_{\sigma}^{k+1}\in f^{-1}\big(U_{k+1}(y_\sigma)\big)$
such that $\partial^X(c_{\sigma}^{k+1})=\phi_k(\partial(\sigma))$.
We define $\phi_{k+1}'(\sigma)=c_{\sigma}^{k+1}$
if $\sigma\not\in L$  and $\phi_{k+1}'(\sigma)=\varphi(\sigma)$ if $\sigma\in L$,
and then extend $\phi_{k+1}'$ to a homomorphism
$\phi_{k+1}:C_{k+1}(K;G)\to S_{k+1}(X;G)$ by linearity.
Condition $(4_{k+1})$ is satisfied with $U^{k+1}_\sigma=U_{k+1}(y_\sigma)$.

In this way we complete the inductive step.
The required full algebraic realization of $C(K;G)$ in $f^{-1}(\mathcal U)$ is
the chain morphism $\phi:C(K;G)\to S(X;G)$ with $\phi=\{\phi_k\}_{k=0}^{n+1}$.
\end{proof}

Proposition 3.1 implies the following \lq\lq approximate lifting" of chain morphisms,
see \cite[Theorem 16.7]{da} for the homotopical version.

\begin{cor}
Let $f:X\to Y$ be as Proposition $3.1$.
Then every open cover $\mathcal U$ of $Y$ has an open refinement $\mathcal V$ covering $Y$
such that:
If $K$ is an $(n+1)$-dimensional simplicial complex,
$L$ its sub-complex  and $\varphi_L:C(L;G)\to S(X;G)$, $\phi:C(K;G)\to S(Y;G)$ are two correct chain morphisms
such that $\phi|C(L;G)=f_\sharp\circ\varphi_L$ and $\phi$ is $\mathcal V$-small,
then  there is a chain morphism $\varphi_K:C(K;G)\to S(X;G)$ extending $\varphi_L$ with $\phi$
and $f_\sharp\circ\varphi_K$ being $\mathcal U$-close.
\end{cor}

\begin{proof}
Let $\mathcal U$ be an open cover of $Y$ and choose $\mathcal U_1$ to be a star-refinement of $\mathcal U$.
Then, there exists a corresponding refinement
$\mathcal V$ of $\mathcal U_1$ satisfying Proposition 3.1.
Let $\phi:C(K;G)\to S(Y;G)$ and $\varphi_L:C(L;G)\to S(X;G)$ be chain morphisms
satisfying the hypotheses of the corollary.
We can suppose that $L$ contains all vertexes of $K$.
Indeed, otherwise for any $v\in K^{(0)}\setminus L$ let $\phi(v)=g\cdot\sigma^0$ with $g\in G$ and
$\sigma^0$ being a singular $0$-simplex in $Y$.
Then $|\sigma^0|$ is a point in $Y$, and define $\varphi_L(v)=g\cdot\tau^0$,
where $\tau^0$ is a singular $0$-simplex in $X$ with $f(|\tau^0|)=|\sigma^0|$.

Because $\phi|C(L;G)=f_\sharp\circ\varphi_L$ and $\phi$ is  $\mathcal V$-small,
$\varphi_L$ is a correct partial algebraic realization of $C(K;G)$ in $f^{-1}(\mathcal V)$.
So, $\varphi_L$ extends to a full algebraic realization $\varphi_K:C(K;G)\to S(X;G)$ of $C(K;G)$
in $f^{-1}(\mathcal U_1)$.
Hence, for every $\sigma\in K$ there exists $U'_\sigma\in\mathcal U_1$
such that  $|\varphi_K(\tau)|\subset f^{-1}(U'_\sigma)$ for all faces $\tau$ of $\sigma$.
On the other hand,
since $\phi$ is $\mathcal V$-small, there exists $V_\sigma\in\mathcal V$  such that
$|\phi(v)|\neq\varnothing$ and $|\phi(v)|\cup |\phi(\tau)|\subset V_\sigma$,
$\tau$ is a face of $\sigma$ and $v\in\sigma^{(0)}$.
Hence, for all $v\in\sigma^{(0)}$ we have
$|\phi(v)|\subset V_\sigma\cap U'_\sigma$.
So, $\rm{St}(V_\sigma;\mathcal U_1)\neq\varnothing$ and
it contains $|\phi(\tau)|\cup f(|\varphi_K(\tau)|)$ for all faces $\tau$ of $\sigma$.
Finally, since $\mathcal V$ refines $\mathcal U_1$ and $\mathcal U_1$ is a star refinement of $\mathcal U$,
there is $U_\sigma\in\mathcal U$ containing all $|\phi(\tau)|\cup f(|\varphi_K(\tau)|)$,
$\tau$ being a face of $\sigma$.
Therefore, $\phi$ is $\mathcal U$-close to $f_\sharp\circ\varphi_K$.
\end{proof}

Let $K$ be a singular sub-complex of $S(P;G)$,
$X$ be a space and $\mathcal U$ an open cover of $X$.
If $L$ is a sub-complex of $K$  containing all $0$-singular simplexes of $K$,
we say that a chain morphism $\varphi:L\to S(X;G)$ is
a {\em partial singular realization} of $K$ in $\mathcal U$
if for every singular simplex $\sigma\in K$ there exists $U_\sigma\in\mathcal U$
such that $|\varphi(\tau)|\subset U_\sigma$ for all faces $\tau$ of $\sigma$ with $\tau\in L$.
If $L=K$, then $\varphi$ is called a {\em full singular realization of $K$ in $\mathcal U$}.
If in the above definition $\varphi(\sigma)$ is a singular $0$-simplex in $S(X;G)$
for every singular $0$-simplex $\sigma$ of $K$,
then $\varphi$ is said to be a {\em correct partial singular realization} of $K$ in $\mathcal U$.

The proof of Proposition 3.1 and Corollary 3.2 remain true
if $\varphi$ is a correct partial singular realization of $K=S^{(n+1)}(P;G)$ in $f^{-1}(\mathcal V)$.
In this case $\varphi$ extends to a full singular realization of $K$ in $f^{-1}(\mathcal U)$.
So, we have the following \lq\lq singular analogues" of Proposition 3.1 and Corollary 3.2.

\begin{pro}
Let $f:X\to Y$ be as in Proposition $3.1$.
Then for every open cover $\mathcal U$ of $Y$
there is an open cover $\mathcal V$ of $Y$ refining $\mathcal U$ such that
any correct partial singular realization of a singular complex $S^{(n+1)}(P;G)$ in $f^{-1}(\mathcal V)$
extends to a full singular realization of $S^{(n+1)}(P;G)$ in $f^{-1}(\mathcal U)$.
\end{pro}

\begin{cor}
Let $f:X\to Y$ be as Proposition $3.1$.
Then every open cover $\mathcal U$ of $Y$ has an open refinement $\mathcal V$ covering $Y$
such that:
If $S^{(n+1)}(P;G)$ is an $(n+1)$-dimensional singular complex,
$L$ its sub-complex  and $\varphi:L\to S(X;G)$,
$\phi:S^{(n+1)}(P;G)\to S(Y;G)$ are two correct chain morphisms
such that $\phi|L=f_\sharp\circ\varphi$ and $\phi$ is $\mathcal V$-small,
then there is a chain morphism $\widetilde{\varphi}:S^{(n+1)}(P;G)\to S(X;G)$ extending $\varphi$
with $\phi$ and $f_\sharp\circ\widetilde{\varphi}$ being $\mathcal U$-close.
\end{cor}

Another application of Proposition 3.3 is
the following chain analogue of Dugundji's extension theorem for $LC^n$-spaces \cite{du1}.
For that reason we introduce the following notion:
a chain morphism $\varphi: S(Z;G)\to S(Y;G)$ is said to be {\em continuous}
if for every $z\in Z$ and any neighborhood $U$ of $|\varphi(z)|$ in $Y$
there is a neighborhood $V$ of $z$ in $Z$ with $|\varphi(z')|\subset U$ for all $z'\in V$
(here, $z$ and $z'$ are treated as singular $0$-simplexes in $S(Z;G)$).
For example,
if $f:Z\to Y$ is a continuous map, then the chain morphism $f_\sharp:S(Z;G)\to S(Y;G)$ is continuous.

\begin{pro}
Let $f:X\to Y$ be as in Proposition $3.1$.
Then each open cover $\mathcal U$ of $Y$ admits an open refinement $\mathcal V$ with the property:
If $A$ is a closed subset of a metric space $M$ and
$\varphi:S(A;G)\to S(X;G)$ is a continuous correct $f^{-1}(\mathcal V)$-small chain morphism,
then there exist a neighborhood $W$ of $A$ in $M$, an open cover $\omega$ of $W$ and
a $f^{-1}(\mathcal U)$-small chain morphism $\widetilde{\phi}:S^{(n+1)}(W,\omega;G)\to S(X;G)$
such that
$\widetilde{\phi}(c)=\varphi(c)$ for all $c\in S^{(n+1)}(W,\omega;G)\cap S(A;G)$.
\end{pro}

\begin{proof}
Every open cover $\mathcal U$ of $Y$ has a refinement $\mathcal V'$ satisfying Proposition 3.3,
and let $\mathcal V$ be an open star refinement of $\mathcal V'$.
Suppose $\varphi:S(A;G)\to S(X;G)$ is a continuous correct $f^{-1}(\mathcal V)$-small morphism,
where $A$ is a closed subset of a metric space $(M,\rho)$.
According to \cite{du1}, there is a locally finite canonical open cover $\alpha$ of $M\setminus A$.
This means that for every $a\in A$ and its neighborhood $O(a)$ in $M$
there exists another neighborhood $\Gamma(a)$ in $M$
such that $\rm{St}\big(\Gamma(a),\alpha\big)\subset O(a)$.
For every $\Lambda\in\alpha$ we choose a point
$a_\Lambda\in A$ with
\begin{enumerate}
\item[(5)] $\rho(a_\Lambda,\Lambda)<2\sup_{z\in \Lambda}\rho(z,A).$
\end{enumerate}
Since $\varphi$ is correct and continuous,
for every $a\in A$ we find $V_a\in\mathcal V$ and $\varepsilon_a>0$
with $|\varphi(z)|\subset f^{-1}(V_a)$ for all $z\in A\cap B_\rho(a,\varepsilon_a)$,
where $B_\rho(a,\varepsilon_a)$ is the open ball in $M$ with radius $\varepsilon_a$ and center $a$.
Using the notations above,
let $\Gamma(a)$ be the corresponding to $O(a)=B_\rho(a,\varepsilon_a/3)$ neighborhood of $a$,
and let $W=\bigcup_{a\in A}\Gamma(a)$ and $\omega=\{\Gamma(a):a\in A\}$.

 \smallskip
\textit{Claim $2$. If $\Lambda\cap\Gamma(a)\neq\varnothing$ for some $\Lambda\in\alpha$ and
$a\in A$, then $a_\Lambda\in B_\rho(a,\varepsilon_a)$.}

Indeed, for any such $\Lambda$ we have $\Lambda\subset B_\rho(a,\varepsilon_a/3)$.
Consequently, $2\sup_{z\in\Lambda}\rho(z,A)<2\varepsilon_a/3$ and, according to $(5)$,
there is $z_\Lambda\in \Lambda$ with $\rho(a_\Lambda,z_\Lambda)<2\varepsilon_a/3$.
So, $\rho(a_\Lambda,a)\leq\rho(a_\Lambda,z_\Lambda)+\rho(z_\Lambda,a)<\varepsilon_a$.

Obviously, $L=S_0(W;G)\cup S^{(n+1)}(A,\omega;G)$ is
a sub-complex of the complex $S^{(n+1)}(W,\omega;G)$
containing all singular $0$-simplexes in $W$.
For any singular $0$-simplex $z\in S_0(W;G)$
we define $\phi_0'(z)=\varphi(z)$ if $z\in A$ and
$\phi'(z)=\varphi(a_{\Lambda(z)})$ if $z\in W\setminus A$,
where $\Lambda(z)$ is an arbitrary element of $\alpha$ containing $z$.
Next,
extend $\phi'_0$ to a homomorphism $\phi_0:S_0(W;G)\to S(X;G)$ by linearity.
Obviously, $\phi_0$ can be extended to a homomorphism $\phi:L\to S(X;G)$
by defining $\phi(\sigma)=\varphi(\sigma)$ for all $\sigma\in S^{(n+1)}(A,\omega;G)\setminus S_0(W;G)$.

 \smallskip
\textit{Claim $3$.
$\phi$ is a correct partial singular realization of $S^{(n+1)}(W,\omega;G)$ in $f^{-1}(\mathcal V')$.}

The correctness of $\phi$ follows from the correctness of $\varphi$.
To show that $\phi$ is a partial singular realization of $S^{(n+1)}(W,\omega;G)$
in $f^{-1}(\mathcal V')$,
let $\sigma\in S^{(n+1)}(W,\omega;G)$ be a singular simplex and $z$ its vertex.
Then $|\sigma|\subset\Gamma(a)\in\omega$ for some $a\in A$,
so  $z$ is a point from $\Gamma(a)$ (we identify $z$ with $|z|$).
Since
$\Gamma(a)\subset B_\rho(a,\varepsilon_a/3)$, according to Claim 2 and the choice of $\varepsilon_a$,
we have $|\phi(z)|=|\varphi(a_{\Lambda(z)})|\subset f^{-1}(V_a)$ if $z\in W\setminus A$.
The inclusion $|\phi(z)|\subset f^{-1}(V_a)$ holds also if $z\in A$.
Therefore,
$|\phi(z)|\subset f^{-1}(V_a)$ for all vertices $z$ of $\sigma$.
Take $V_\sigma\in\mathcal V'$ with $\rm{St}(V_a,\mathcal V)\subset V_\sigma$
(recall that $\mathcal V$ is a star refinement of $\mathcal V'$)
and let $\tau\in L$ be a face of $\sigma$.
If $\tau$ is $0$-simplex, then $|\phi(\tau)|$ (being a subset of $f^{-1}(V_a)$) is contained
in $f^{-1}(V_\sigma)$.
If $\tau$ is a singular simplex of dimension $\geq 1$,
then $\tau\in S^{(n+1)}(A,\omega;G)$.
Because $\varphi$ is $f^{-1}(\mathcal V)$-small, there is $V_\tau\in\mathcal V$ such that
$|\varphi(z)|\cup |\varphi(\tau)|\subset f^{-1}(V_\tau)$ for all vertices $z$ of $\tau$.
So, $|\varphi(z)|\subset f^{-1}(V_a)\cap f^{-1}(V_\tau)$ for any vertex $z$ of $\tau$.
Consequently, $\rm{St}(V_a,\mathcal V)\neq\varnothing$ and contains $|\phi(\tau)|$.
Hence,
$|\phi(\tau)|\subset f^{-1}(V_\sigma)$ for all faces $\tau$ of $\sigma$.
Thus,
$\phi$ is a correct partial singular realization of $S^{(n+1)}(W,\omega;G)$ in $f^{-1}(\mathcal V')$.

Finally, by Proposition 3.3,
$\phi$ extends to a full singular realization $\widetilde{\phi}$ of $S^{(n+1)}(W,\omega;G)$ in $f^{-1}(\mathcal U)$.
Therefore, $\widetilde{\phi}$ is $f^{-1}(\mathcal U)$-small and $\widetilde{\phi}(c)=\varphi(c)$
for any $c\in S^{(n+1)}(W,\omega;G)\cap S(A;G)$.
\end{proof}

If in Proposition 3.5 $\dim M\leq n+1$,
then we have the following \lq\lq approximate extension" version of Proposition 3.5

\begin{pro}
Suppose $f$ is as in Proposition $3.1$.
Then for every open cover $\mathcal U$ of $Y$
there exists an open cover $\mathcal V$ of $Y$ with the following property:
If $\varphi:S(A;G)\to S(X;G)$ is a continuous correct $f^{-1}(\mathcal V)$-small chain morphism,  where $A$ is a closed subset of a metric space $M$ with $\dim M\leq n+1$,
then there exist an open set $W\subset M$ containing $A$, an open cover $\alpha$ of $W$
and a correct chain morphism
$\phi:S(W,\alpha;G)\to S(X;G)$ such that
$\phi|S(A,\alpha;G)$ and $\varphi|S(A,\alpha;G)$ are $f^{-1}(\mathcal U)$-close.
\end{pro}

\begin{proof}
Let $\mathcal U$ be an open cover of $Y$ and
$\mathcal U_1$ be a star refinement of $\mathcal U$.
Take another open cover $\mathcal V_1$ of $Y$ satisfying the hypotheses of Proposition 3.1
with respect to $\mathcal U_1$
(i.e., any correct partial algebraic realization of $C(K;G)$ in $f^{-1}(\mathcal V_1)$
can be extended to a full algebraic realization of $C(K;G)$ in $f^{-1}(\mathcal U_1)$,
where $K$ is a simplicial complex with $\dim K\leq n+1$).
Let $\mathcal V$ be a locally finite open star-refinement of $\mathcal V_1$.
Since $\varphi$ is continuous and correct,
for every $a\in A$ there are $V_a\in\mathcal V$ and a neighborhood $O_a$ of $a$ in $A$
with $|\varphi(z)|\subset f^{-1}(V_a)$ for any $z\in O_a$.
Take a locally finite open cover $\Gamma=\{\Lambda_t:t\in T\}$ of $A$
refining the cover $\{O_a:a\in A\}$
such that the nerve of $\Gamma$ is of dimension $\leq n+1$.
Because $M$ is a metric space,
we can extend each $\Lambda_t\in\Gamma$ to an open set $\widetilde\Lambda_t\subset M$
such that for any finitely many $\Lambda_{t_1},..,\Lambda_{t_k}$
we have $\bigcap\widetilde\Lambda_{t_i}\neq\varnothing$
if and only if $\bigcap\Lambda_{t_i}\neq\varnothing$.
The last relation implies that the nerve $K$ of
$\widetilde{\Gamma}=\{\widetilde\Lambda_{t}:t\in T\}$ is also of dimension $\leq n+1$.
Let $W=\bigcup\{\widetilde\Lambda_{t}:t\in T\}$.
For every
$\widetilde\Lambda_{t}\in\widetilde{\Gamma}$ choose a point $a(t)\in\Lambda_t$ and
define $\psi_0'(\widetilde{\Lambda_t})=\varphi(a(t))$,
where
$\widetilde{\Lambda_t}$ is considered as a vertex of $K$ and $a(t)$
as a singular $0$-simplex from $S(A;G)$.
Then extend $\psi_0'$ to a homomorphism $\psi_0:C_0(K;G)\to S_0(X;G)$.

Since $\mathcal V$ is a star-refinement of $\mathcal V_1$,
$\psi_0$ is a correct partial algebraic realization of $C(K;G)$ in $\mathcal V_1$.
Indeed suppose
$\sigma=(\widetilde\Lambda_{t_0},..,\widetilde\Lambda_{t_m})$ is a simplex from $K$.
Then for every $i=0,..,m$ there is $V_i\in\mathcal V$ such that
$|\varphi(a)|\in f^{-1}(V_i)$ for all $a\in\Lambda_{t_i}$.
So, $f^{-1}(V_0)\cap f^{-1}(V_i)\neq\varnothing$ and
$|\varphi(a(t_i))|\subset f^{-1}(V_i)$, $i=0,..,m$. Hence, $\rm{St}(f^{-1}(V_0),f^{-1}(\mathcal V))$
contains $\bigcup|\varphi(a(t_i))|$.
Consequently, $\bigcup|\varphi(a(t_i))|\subset f^{-1}(V')$ for some $V'\in\mathcal V_1$
(recall that $\mathcal V$ is a star-refinement of $\mathcal V_1$).
Thus, $\psi_0$ is a correct partial algebraic realization of $C(K;G)$ in $\mathcal V_1$.

So, $\psi_0$ extends to a full algebraic realization
$\psi:C(K;G)\to S(X;G)$ of $C(K;G)$ in $f^{-1}(\mathcal U_1)$.
Let $\kappa: W\to |K|$ be a canonical map,
where the polytope $|K|$ is equipped with the Whitehead topology.
According to \cite[Proposition 8.6.6]{hw},
there are an open cover $\mathcal S$ of $|K|$ such that each $|s|$, $s\in K$,
is contained in some $P_s\in\mathcal S$,
and a chain equivalence $\gamma:S(|K|,\mathcal S;G)\to C^\Omega(K;G)$.
Here $C^\Omega(K;G)$ is the chain complex whose simplexes are finite arrays
$[\widetilde\Lambda_0,\widetilde\Lambda_1,..,\widetilde\Lambda_k]$,
where all $\widetilde\Lambda_i$,
not necessarily distinct,
are vertices of $K$ spanning a simplex from $K$.
There exists also a natural chain morphism $\theta:C^\Omega(K;G)\to C(K;G)$ such that
$\theta([\widetilde\Lambda_0,\widetilde\Lambda_1,..,\widetilde\Lambda_k])$
is the simplex
$(\widetilde\Lambda_0,\widetilde\Lambda_1,..,\widetilde\Lambda_k)\in K$
if all $\widetilde\Lambda_i$ are distinct, and $0$ otherwise.
Let $\alpha$ be the intersection of the covers $\widetilde\Gamma$ and
$\kappa^{-1}(\mathcal S)$,
and let $\phi_1:S(W,\alpha;G)\to C(K;G)$ and
$\phi:S(W,\alpha;G)\to C(K;G)$ be the chain morphisms
$\phi_1=\theta\circ\gamma\circ\kappa_\sharp$ and $\phi=\psi\circ\phi_1$, respectively.

Let show that $\varphi|S(A,\alpha;G)$ and $\phi|S(A,\alpha;G)$ are $f^{-1}(\mathcal U)$-close. Indeed,
since $\varphi$ is $f^{-1}(\mathcal V)$-small,
for any singular simplex $\sigma\in S(A,\alpha;G)$ there is $V_\sigma\in\mathcal V$
with $|\varphi(\tau)|\subset f^{-1}(V_\sigma)$ for all faces $\tau$ of $\sigma$.
On the other hand,
$\sigma_1=\kappa_\sharp(\sigma)$ is a singular simplex from $S(|K|,\mathcal S;G)$
such that, according to the definition of $\gamma$ (see \cite[pp. 339]{hw}),
$\gamma(\sigma_1)$ is a \lq\lq simplex"
$s=[\widetilde\Lambda_0,\widetilde\Lambda_1,..,\widetilde\Lambda_k]$ from $C^\Omega(K;G)$
satisfying the following condition:
if $\tau$ is a face of $\sigma$,
then $\kappa_\sharp(\tau)$ is a face of $\sigma_1$ and
the vertices of $\gamma(\kappa_\sharp(\tau))$ are also
vertices of $\gamma(\kappa_\sharp(\sigma))$.
In particular, for any vertex $v$ of $\sigma$ we have
$\gamma(\kappa_\sharp(v))=\gamma(\kappa_\sharp(|v|))$
is one of the vertexes $\widetilde\Lambda_i$
such that $|v|$ is a point from $\Lambda_i$.
So, for every face $\tau$ of $\sigma$ either $\phi_1(\tau)=0$ or
$\phi_1(\tau)$ is a simplex from $K$ whose vertices are contained in the set
$\{\widetilde\Lambda_i;i=0,1,..k\}$,
but definitely the union of all $\phi(\tau)$,
$\tau$ is a face of $\sigma$, is non-empty.
Hence,
there exists a simplex $\delta\in K$  containing $\phi_1(\tau)$ for all faces $\tau$ of $\sigma$
such that the vertices of $\delta$ are in the set $\{\widetilde\Lambda_i;i=0,1,..k\}$.
Since $\psi$ is $f^{-1}(\mathcal U_1)$-small,
we can find $U_\delta\in\mathcal U_1$ containing all $|\phi(\tau)|\subset f^{-1}(U_\delta)$,
$\tau$ is a face of $\sigma$.
We fix a vertex $v^*$ of $\sigma$.
Then $\phi_1(v^*)=\widetilde\Lambda_j$ for some $j$ with $|v^*|\in\Lambda_j$,
and let $V_a\in\mathcal V$ such that $|\varphi(z)|\in f^{-1}(V_a)$ for all $z\in\Lambda_j$.
So, according to the definition of $\psi$, $\phi(v^*)=\psi(\widetilde\Lambda_j)=\varphi(z^*)$
for some $z^*\in\Lambda_j$.
Consequently,
$|\phi(v^*)|=|\varphi(z^*)|\in f^{-1}(V_a)$.
Hence, $|\phi(v^*)|\in f^{-1}(V_a)\cap f^{-1}(U_\sigma)$ and
$|\varphi(v^*)|\in f^{-1}(V_a)\cap f^{-1}(V_\sigma)$.
Therefore, since $\mathcal V$ is refining $\mathcal U_1$, for all faces $\tau$ of $\sigma$
we have
$$
|\varphi(\tau)|\cup |\phi(\tau)|\subset f^{-1}(V_\sigma)\cup f^{-1}(U_\sigma)\subset
\rm{St}(f^{-1}(V_a),f^{-1}(\mathcal U_1)).
$$

\noindent
Because  $\rm{St}(f^{-1}(V_a),f^{-1}(\mathcal U_1))$ is contained in $f^{-1}(U)$
for some $U\in\mathcal U$,
we finally obtain that $\varphi|S(A,\alpha;G)$ and $\phi|S(A,\alpha;G)$ are
$f^{-1}(\mathcal U)$-close.
\end{proof}

Next proposition is an analogue of Proposition 2.4.

\begin{pro}
Let $f:X\to Y$ be as in Proposition $3.1$, $Z$ an arbitrary space and $A\subset Z$ a closed subset. Then for every open cover $\mathcal U$ of $Y$
there exists an open refinement $\mathcal V$ of $\mathcal U$ such that
for any two correct $f^{-1}(\mathcal V)$-close chain morphisms
$\varphi,\phi:S^{(n)}(Z;G)\to S(X;G)$
and any $f^{-1}(\mathcal V)$-small chain homotopy
$\Phi:S(A;G)\to S(X;G)$ between $\varphi|S^{(n)}(A;G)$ and $\phi|S^{(n)}(A;G)$
there exists a  $f^{-1}(\mathcal U)$-small homotopy
$D:S^{(n)}(Z;G)\to S(X;G)$ between $\varphi$ and $\phi$ extending $\Phi$;
\end{pro}

\begin{proof}
We follow the proof of \cite[Lemma 5.4]{bo}.
As in the proof of Proposition 3.1, for every $k=n,n-1,..,0$
we construct open covers $\mathcal U_k$ and $\mathcal V_k$ of $Y$
such that $\mathcal U=\mathcal U_n$, $\mathcal U_k$ star-refines $\mathcal V_{k+1}$
for $k=n-1,..,0$ and for each $V\in\mathcal V_k$
there exists
$U\in\mathcal U_k$ with $f^{-1}(V)\stackrel{H_k}{\hookrightarrow}f^{-1}(U)$ if $k=n,..,1$
and $\rm{St}(f^{-1}(V),f^{-1}(\mathcal V_0))\stackrel{H_0}{\hookrightarrow}f^{-1}(U)$ if $k=0$.
We claim that $\mathcal V=\mathcal V_0$ is the required cover.
Indeed,
suppose $\varphi,\phi:S^{(n)}(Z;G)\to S(X;G)$ are two $f^{-1}(\mathcal V)$-close correct chain
morphisms and $\Phi:S(A;G)\to S(X;G)$ is
a $f^{-1}(\mathcal V)$-small chain homotopy between
$\varphi|S^{(n)}(A;G)$ and $\phi|S^{(n)}(A;G)$.
We are going to construct
homomorphisms $D_k:S_k(Z;G)\to S_{k+1}(X;G)$ such that:

\begin{itemize}
\item[(6)] $\partial D_k(c^k)=\varphi(c^k)-\phi(c^k)-D_{k-1}(\partial c^k)$
for all $c^k\in S_k(Z;G)$;
\item[(7)] $D_k(c^k)=\Phi(c^k)$ for all $c^k\in S(A;G)$;
\item[(8)] For any singular $k$-simplex $\sigma\in S_{k}(Z;G)$
there is $U_\sigma\in\mathcal U_k$ such that $f^{-1}(U_\sigma)$ contains
$|D_i(\tau)|\cup|\phi(\tau)|$ for all $i\leq k$ and all $i$-dimensional faces $\tau$ of $\sigma$.
\end{itemize}

\noindent
Because $\varphi$ and $\phi$ are $f^{-1}(\mathcal V)$-close,
for any $z\in Z$ there is $V_z\in\mathcal V$ with
$\varphi(z),\phi(z)$ being singular $0$-simplexes in $f^{-1}(V_z)$
(we identify each $z\in Z$ with the singular $0$-simplex $\sigma\in S_0(Z;G)$
such that $|\sigma|=\{z\}$).
So, $\varphi(z)-\phi(z)$ is a singular $0$-cycle in $f^{-1}(V_z)$.
Since
$\rm{St}(f^{-1}(V_z), f^{-1}(\mathcal V))\stackrel{H_0}{\hookrightarrow}f^{-1}(U_z)$
for some $U_z\in\mathcal U_0$,
there is $c^1_z\in S_1(f^{-1}(U_z);G)$ with $\partial c^1_z=\varphi(z)-\phi(z)$.
For every $z\in Z$ we define
$D_0'(z)=c^1_z$ if $z\not\in A$ and $D_0'(z)=\Phi(z)$ if $z\in A$,
and extend $D_0'$ linearly to a homomorphism $D_0:S_0(Z;G)\to S_1(X;G)$.
Obviously, $|D_0(z)|\cup |\phi(z)|\subset f^{-1}(U_z)$ if $z\not\in A$. If $z\in A$,
then there is $V_z'\in\mathcal V$ with $|\Phi(z)|\cup |\phi(z)|\subset f^{-1}(V_z')$
(recall that $\Phi$ is $f^{-1}(\mathcal V)$-small).
So, $|\phi(z)|\subset f^{-1}(V_z')\cap f^{-1}(V_z)$,
which shows that $\rm{St}(f^{-1}(V_z),f^{-1}(\mathcal V))\neq\varnothing$
and contains $|\Phi(z)|\cup |\phi(z)|$.
Thus, $|\Phi(z)|\cup |\phi(z)|\subset f^{-1}(U_z)$ for all $z\in Z$.
Therefore, $D_0$ satisfies  conditions $(6) - (8)$.

Suppose we already constructed the homomorphisms $D_i:S_i(Z;G)\to S_{i+1}(X;G)$, $i\leq k$,
satisfying the above conditions,
and let $\sigma$ be a singular $(k+1)$-simplex from $S_{k+1}(Z;G)$.
Since $\varphi$ and $\phi$ are $f^{-1}(\mathcal V)$-close,
there exists $V_\sigma\in\mathcal V$ such that
$|\phi(z)|\cup |\varphi(\tau)|\cup |\phi(\tau)|\subset f^{-1}(V_\sigma)$
for all faces $\tau$ and all vertexes $z$ of $\sigma$.
On the other hand, according to $(8)$, for any $k$-singular face $\tau$ of $\sigma$
there is $U_\tau^k\in\mathcal U_k$ with $f^{-1}(U_\tau^k)$ containing $|D_i(s)|\cup|\phi(z)|$
for all $i\leq k$ and all $i$-dimensional faces $s$ of $\tau$ and $z\in\tau^{(0)}$.
So, $|\phi(z)|\subset f^{-1}(V_\sigma)\cap f^{-1}(U_\tau^k)$ for all $k$-faces $\tau$ of $\sigma$
and all $z\in\tau^{(0)}$. Hence,
$\rm{St}(f^{-1}(V_\sigma);f^{-1}(\mathcal U_k))\neq\varnothing$ and contains $|\gamma_\sigma|$
and all $|D_i(s)|\cup|\phi(s)|$,
$i\leq k$ and $s$ is a $i$-dimensional face of $\sigma$,
where
$\gamma_\sigma=\varphi(\sigma)-\phi(\sigma)-D_k(\partial\sigma)$.
Choose $V_\sigma^{k+1}\in\mathcal V_{k+1}$ and $U_\sigma^{k+1}\in\mathcal U_{k+1}$
such that $\rm{St}(f^{-1}(V_\sigma);f^{-1}(\mathcal U_k))\subset f^{-1}(V_\sigma^{k+1})$
and $f^{-1}(V_\sigma^{k+1})\stackrel{H_{k+1}}{\hookrightarrow}f^{-1}(U_\sigma^{k+1})$.
Finally, since $\gamma_\sigma$ is a singular $(k+1)$-cycle in $f^{-1}(V_\sigma^{k+1})$,
we can find a $(k+2)$-chain $c^{k+2}_\sigma\in S_{k+2}(f^{-1}(U_\sigma^{k+1});G)$
with $\partial c^{k+2}_\sigma=\gamma_k$.
Define $D_{k+1}'(\sigma)=c^{k+2}_\sigma$ if $\sigma\not\in S_{k+1}(A;G)$ and
$D_{k+1}'(\sigma)=\Phi(\sigma)$ if $\sigma\in S_{k+1}(A;G)$,
and extend $D_{k+1}'$ linearly to a homomorphism $D_{k+1}:S_{k+1}(Z;G)\to S_{k+1}(X;G)$
satisfying conditions $(6) - (8)$.

In this way
we construct the homomorphisms $D_k$ for all $k\leq n$ satisfying conditions $(6) - (8)$.
Then $D=\{D_k\}_{k\leq n}$ is the required homotopy
between $\varphi$ and $\phi$ extending $\Phi$.
\end{proof}

We also have the following proposition, whose proof is similar to that one of Proposition 3.7.

\begin{pro}
Let $f:X\to Y$ be as in Proposition $3.1$, $K$ a simplicial complex with $\dim K\leq n$ and
$L$ a sub-complex of $K$.
Then for every open cover $\mathcal U$ of $Y$
there exists an open refinement $\mathcal V$ of $\mathcal U$ such that
for any two correct $f^{-1}(\mathcal V)$-close chain morphisms $\varphi,\phi:C(K;G)\to S(X;G)$
and any $f^{-1}(\mathcal V)$-small chain homotopy
$\Phi:C(L;G)\to S(X;G)$ between $\varphi|C(L;G)$ and $\phi|C(L;G)$
there exists a  $f^{-1}(\mathcal U)$-small homotopy
$D:C(K;G)\to S(X;G)$ between $\varphi$ and $\phi$ extending $\Phi$.
\end{pro}


\section{Homologically locally connected spaces}

First, let us note that all results from Section 3 remain true
in case $X$ is an $lc^n_G$-space and $f:X\to X$ is the identity map.
Some of these results characterize $lc^n_G$-spaces.
For example, we have the following proposition.

\begin{pro}
A paracompact space $X$ is $lc^n_G$ if and only if
each open cover $\mathcal U$ of $X$ has an open refinement $\mathcal V$
such that for any two correct $\mathcal V$-close chain morphisms
$\varphi,\phi:S^{(n)}(Z;G)\to S(X;G)$,
where $Z$ is an arbitrary space,
there exists an  $\mathcal U$-small homotopy
$D:S^{(n)}(Z;G)\to S(X;G)$ between $\varphi$ and $\phi$.
\end{pro}

\begin{proof}
The necessity follows from Proposition 3.7.
So, we need to prove only the sufficiency.
Suppose $X$ satisfies that condition, and let $U_x$ be a neighborhood of a point $x\in X$ and
$\mathcal U=\{U_x,X\setminus\overline W_x\}$,
where $W_x$ is a neighborhood of $x$ with $\overline W_x\subset U_x$.
Then there is an open cover $\mathcal V$ of $X$ satisfying the hypotheses of the proposition.
We can assume that $\mathcal V$ is a star-refinement of $\mathcal U$,
and take $V_x\in\mathcal V$ containing $x$.
Obviously, $\rm{St}(V_x,\mathcal V)\subset U_x$.
Consider the correct chain morphisms
$\varphi,\phi:S^{(n)}(V_x;G)\to S(X;G)$
defined by $\varphi(c)=c$ and
$\phi(\sigma^k)=\sigma_x^k$ for all $c\in S^{(n)}(V_x;G)$ and all singular $k$-simplexes
$\sigma^k\in S^{(n)}(V_x;G)$,
where $\sigma_x^k$ denotes the unique singular $k$-simplex with $|\sigma_x^k|=\{x\}$.
Then there exists a $\mathcal U$-small homotopy
$D:S^{(n)}(V_x;G)\to S(X;G)$ between $\varphi$ and $\phi$.
Let $c^k=\sum g_i\sigma^k_i\in S^{(n)}(V_x;G)$ be a $k$-cycle, $k\leq n$.
Hence, $D(c^k)$ is a chain from $S_k(X;G)$ such that $\partial D(c^k)=c^k-\phi(c^k)$.
Define $c^{k+1}=D(c^k)+(\sum g_i)\sigma_x^{k+1}$.
So,
$\partial c^{k+1}=c^k-(\sum g_i)\sigma_x^k+(\sum g_i)\partial\sigma_x^{k+1}$.
When $k+1$ is an odd integer,
we have $\partial\sigma_x^{k+1}=\sigma_x^k$.
Therefore, in this case $\partial c^{k+1}=c^k$.
For even integers $k+1$ we have
$\partial\sigma_x^{k+1}=0$ and $\partial\sigma_x^k=\sigma_x^{k-1}$.
Then, since $c^k$ is a cycle,
$0=\partial\phi(c^k)=(\sum g_i)\sigma_x^{k-1}$.
Consequently, $\sum g_i=0$ and $\partial c^{k+1}=c^k$.
Therefore, $\partial c^{k+1}=c^k$ for all integers $k$.

It remains to see that $|c^{k+1}|\subset U_x$.
To this end, let $\sigma_j^k$ be a fixed singular simplex from the representation of $c^k$.
Since $D$ is $\mathcal U$-small,
$|D(\sigma_j^k)|\cup |\phi(v)|$ is contained in an element of $\mathcal U$
for every vertex $v$ of $\sigma_j^k$.
But $|\phi(v)|=|\sigma_x^0|=\{x\}$,
so $|D(\sigma_j^k)|\subset U_x$ for all $j$.
Because $|(\sum g_i)\sigma_x^{k+1}|=\{x\}$, we finally conclude that
$|c^{k+1}|\subset U_x$.
\end{proof}

Here is another property of $lc^n_G$-spaces,
similar to the corresponding property for $LC^n$-spaces (see \cite[Theorem 6.1]{hu}).

\begin{pro}
Let $X$ be a paracompact $lc^n_G$-space.
Then for each open cover $\mathcal U$ of $X$
there exists a simplicial complex $K$ of dimension $\leq n+1$ together
with a correct chain morphism $\Phi:C(K;G)\to S(X;G)$
such that for every correct continuous $\mathcal V$-small chain morphism
$\varphi:S(Y;G)\to S(X;G)$, where $Y$ is a paracompact space with $\dim Y\leq n+1$,
there exist an open cover $\Upsilon$ of $Y$ and
a chain morphism $\phi:S(Y,\Upsilon;G)\to C(K;G)$ such that $\varphi|S(Y,\Upsilon;G)$ and
$(\Phi\circ\phi)$ are $\mathcal U$-close.
\end{pro}

\begin{proof}
Let $\mathcal U$ be a given open cover of $X$ and
$\mathcal U_1$ be a star open refinement of $\mathcal U$.
Then there is an open cover $\mathcal V_1$ of $X$
satisfying the hypotheses of Proposition 3.1
(with $X=Y$, $\mathcal U=\mathcal U_1$ and $f$ being the identity).
Let $\mathcal V$ be a locally finite star-refinement of $\mathcal V_1$ and
$K$ be the $(n+1)$-dimensional skeleton of the nerve  of $\mathcal V$
(we consider $K$ as a simplicial complex, not as a polytope).
For each $V\in\mathcal V$ pick a point $x_v\in V$ and
define $\Phi_0'(V)=x_v$ ($V$ is considered here as a vertex of $K$),
and extend $\Phi_0'$ to a homomorphism $\Phi_0:C_0(K;G)\to S_0(X;G)$.
Because $\mathcal V$ is a star-refinement of $\mathcal V_1$,
$\Phi_0$ is a correct partial algebraic realization of $C(K;G)$ in $\mathcal V_1$.
So, by Proposition 3.1,
$\Phi_0$ extends to a full algebraic realization $\Phi:C(K;G)\to S(X;G)$ in $\mathcal U_1$.

Suppose now that $Y$ is a paracompact space of dimension $\leq n+1$ and
$\varphi:S(Y;G)\to S(X;G)$ is a continuous correct $\mathcal V$-small morphism.
Then
$\varphi(y)$ is a singular $0$-simplex in $S(X;G)$,
so $|\varphi(y)|$ is a point and it is contained in some $V_y\in\mathcal V$.
Since $\varphi$ is continuous,
there is a neighborhood $\Lambda_y$ of $y$ in $Y$
such that $|\varphi(z)|\subset V_y$ for all $z\in\Lambda_y$.
In this way we obtain an open cover $\Gamma=\{\Lambda_y:y\in Y\}$ of $Y$.
Because $\dim Y\leq n+1$,
we can suppose that
$\Gamma$ is locally finite and its nerve $\mathcal N_\Gamma$ is at most $(n+1)$-dimensional.
So, there exists a simplicial map $\lambda:\mathcal N_\Gamma\to K$
defined by the assignment $\Lambda_y\mapsto V_y$, $y\in Y$,
and let $\kappa:Y\to|\mathcal N_\Gamma|$ be a canonical map ($|\mathcal N_\Gamma|$ is
equipped with the Whitehead topology).
According to \cite[Proposition 8.6.6]{hw},
there are an open cover $\mathcal S$ of $|\mathcal N_\Gamma|$
such that each $|s|$, $s\in\mathcal N_\Gamma$, is contained in some $P_s\in\mathcal S$,
and a chain equivalence
$\gamma:S(|\mathcal N_\Gamma|,\mathcal S;G)\to C^\Omega(\mathcal N_\Gamma;G)$.
Here $C^\Omega(\mathcal N_\Gamma;G)$ is
the chain complex whose simplexes are finite arrays
$[\Lambda_0,\Lambda_1,..,\Lambda_k]$,
where all $\Lambda_i$, not necessarily distinct,
are vertices of $\mathcal N_\Gamma$ spanning a simplex from $\mathcal N_\Gamma$.
There exists also a natural chain morphism
$\theta:C^\Omega(\mathcal N_\Gamma;G)\to C(\mathcal N_\Gamma;G)$
such that $\theta([\Lambda_0,\Lambda_1,..,\Lambda_k])$ is the simplex
$(\Lambda_0,\Lambda_1,..,\Lambda_k)\in C(\mathcal N_\Gamma;G)$
if all $\Lambda_i$ are distinct,  and $0$ otherwise.
Let $\Upsilon$ be the intersection of the covers $\Gamma$ and $\kappa^{-1}(\mathcal S)$,
and let $\phi:S(Y,\Upsilon;G)\to C(K;G)$ be
the chain morphism $\phi=\lambda_\sharp\circ\theta\circ\gamma\circ\kappa_\sharp$.

It remains to show that $\varphi|S(Y,\Upsilon;G)$ and
$(\Phi\circ\phi)$ are $\mathcal U$-close.
We follow the final part of the proof of Proposition 3.6.
Let $\sigma\in S(Y,\Upsilon;G)$ be a singular simplex.
Since $\varphi$ is $\mathcal V$-small,
there is $V_\sigma\in\mathcal V$ containing $|\varphi(\tau)|$ for all faces $\tau$ of $\sigma$.
On the other hand,
$\sigma_1=\kappa_\sharp(\sigma)$ is a singular simplex
from $S(|\mathcal N_\Gamma|,\mathcal S;G)$ such that,
according to the definition of $\gamma$ (see \cite[pp. 339]{hw}),
$\gamma(\sigma_1)$ is a \lq\lq simplex" $s=[\Lambda_0,\Lambda_1,..,\Lambda_k]$
from $C^\Omega(\mathcal N;G)$ satisfying the following condition:
if $\tau$ is a face of $\sigma$, then
$\kappa_\sharp(\tau)$ is a face of $\sigma_1$ and
the vertices of $\gamma(\kappa_\sharp(\tau))$ are also
vertices of $\gamma(\kappa_\sharp(\sigma))$.
In particular, for any vertex $v$ of $\sigma$ we have
$\gamma(\kappa_\sharp(v))=\gamma(\kappa_\sharp(|v|))$ is one of the vertexes $\Lambda_i$
such that $|v|$ is a point from $\Lambda_i$.
So, for every face $\tau$ of $\sigma$ either $\phi(\tau)=0$ or
$\phi(\tau)$ is a simplex from $K$ whose vertices are contained in the set
$\{\lambda(\Lambda_i);i=0,1,..k\}$,
but definitely the union of all $\phi(\tau)$,
$\tau$ is a face of $\sigma$, is non-empty.
Hence,
there exists a simplex $\delta\in K$  containing $\phi(\tau)$ for all faces $\tau$ of $\sigma$
such that the vertices of $\delta$ are in the set $\{\lambda(\Lambda_i);i=0,1,..k\}$.
Since $\Phi$ is $\mathcal U_1$-small,
we can find $U_\delta\in\mathcal U_1$ containing all
$|\Phi(\phi(\tau))|\subset U_\delta$, $\tau$ is a face of $\sigma$.
We fix a vertex $v^*$ of $\sigma$.
Then
$\phi(v^*)=\lambda(\Lambda_{j})$ for some $0\leq j\leq k$ with $|v^*|\in\Lambda_{j}$,
and $\varnothing\neq |\Phi(\phi(v^*))|\subset U_\delta$.
But $\Phi(\lambda(\Lambda_{j}))$ is a singular $0$-simplex from $S(X,G)$
whose carrier is a point $x^*\in\lambda(\Lambda_{j})$.
Consequently, according to the definition of the sets
$\Lambda_y$, we have $|\varphi(v^*)|\in\lambda(\Lambda_{j})$.
Therefore, $x^*\in U_\delta\cap \lambda(\Lambda_{j})$ and
$|\varphi(v^*)|\in V_\sigma\cap\lambda(\Lambda_{j})$ with
$\lambda(\Lambda_{j}), V_\sigma\in\mathcal V$ and  $U_\delta\in\mathcal U_1$.
Since $V_\sigma$ is contained in some element of $\mathcal U_1$,
we have that
$$
|\varphi(\tau)|\cup|\Phi(\phi(\tau))|
\subset  U_\delta\cup V_\sigma\cup\lambda(\Lambda_{j})
\subset\rm{St}\big(\lambda(\Lambda_{j}),\mathcal U_1\big)
$$
for all faces $\tau$ of $\sigma$.
Finally, since $\mathcal U_1$ is a star refinement of $\mathcal U$,
there is $U\in\mathcal U$ containing
$\rm{St}\big(\lambda(\Lambda_j,\mathcal U_1\big)$.
\end{proof}

Because every $n$-dimensional metric $LC^n$-space is an $ANR$,
it is interesting if $n$-dimensional metric $lc^n_G$-spaces are algebraic $ANR_G$.
We still do not know whether this is true,
but we can show that any such space has a weaker property.

\begin{dfn}
We say that a metric space $X$ is an {\em approximate absolute neighborhood $G$-retract}
(briefly, {\em algebraic $AANR_G$})
if for every embedding of $X$ as a closed subset of a metric space $Y$ and
every open cover $\mathcal U$ of $X$
there is a neighborhood $W$ of $X$ in $Y$,
an open cover $\alpha$ of $W$ and a chain morphism $\phi:S(W,\alpha;G)\to S(X;G)$
such that $\phi|S(X,\alpha;G)$ and the identity morphism on $S(X,\alpha;G)$
are $\mathcal U$-close.
The morphism $\phi$ is called an algebraic approximate $\mathcal U$-retraction.
\end{dfn}

\begin{pro}
Any $n$-dimensional metric $lc^n_G$-space is an $AANR_G$.
\end{pro}

\begin{proof}
Let $X$ be a metric $lc^n_G$-space and $\mathcal U$ be an open cover of $X$.
By \cite{ol} $X$ can be embedded as subset (not necessarily a closed) of
an $(n+1)$-dimensional metrizable $AR$-space $Z$.
According to Proposition 3.6,
there exist an open set $W_Z\subset Z$ containing $X$, an open cover $\beta$ of $W_Z$ and
a chain morphism $\phi:S(W_Z,\beta;G)\to S(X;G)$
such that $\phi|S(X,\beta;G)$ and the identity on $S(X,\beta;G)$ are $\mathcal U$-close.
Now, assume $X$ is a closed subset of a metric space $Y$ and $r:Y\to Z$ is a map
extending the identity on $X$ (such $r$ exists because $Z$ is an $AR$).
Let  $W=r^{-1}(W_Z)$ and $\alpha=r^{-1}(\beta)$.
Then $\phi\circ r_\sharp:S(W,\alpha;G)\to S(X;G)$ is an algebraic approximate
$\mathcal U$-retraction.
\end{proof}

It is well known that if $f\colon X\to Y$ is a closed homotopically $UV^n$-surjection
between metric spaces,
then $Y$ is $LC^n$, see \cite{bog}, \cite{da}, \cite{du}.
The question  whether the homological version of this result is also true is very natural.
It is easily seen that
$Y$ is $lc^0_G$ provided $f$ is a closed homologically $UV^0_G$-surjection
between paracompact spaces.
We can show that $Y$ has the an \lq\lq approximate version" of the $lc^n_G$-property
if $f$ is a closed homologically $UV^n_G$-surjection.

\begin{dfn}
A space $X$ has the {\em approximate $lc^n_G$-property}
if for every $x\in X$ and its neighborhood $U_x$ in $X$
there exist two neighborhoods $V_x\subset W_x$
such that for every  cycle $c^k\in S_k(V_x;G)$, $k\leq n$,
there exists a cycle $\widetilde c^k\in S_k(W_x;G)$
whose vertices are the same as of $c^k$
and homologous to zero in $U_x$.
\end{dfn}

\begin{pro}
Let $f:X\to Y$ be a closed homologically $UV^n(G)$-surjection between paracompact spaces.
Then $Y$ is approximately $lc^n_G$.
\end{pro}

\begin{proof}
Let $y\in Y$ and $U_y\subset Y$ be a neighborhood of $y$.
Choose another two neighborhoods $W_y$ and $O_y$ of $y$ such that
$\overline O_y\subset W_y$ and $f^{-1}(W_y)\stackrel{H_{k}}{\hookrightarrow}f^{-1}(U_y)$
for all $k\leq n$. Then $\mathcal U_1=\{W_y,Y\setminus\overline O_y\}$
is an open cover of $Y$ and let $\mathcal U$ be a star-refinement of $\mathcal U_1$.
There exists another open cover $\mathcal V$ of $Y$ refining $\mathcal U$ and
satisfying the hypotheses of Corollary 3.4. Let $V_y$ be an element of
$\mathcal V$ containing $y$.
Obviously, $\rm{St}(V_y,\mathcal U)\subset W_y$.
Consider $L=S_0(V_y;G)$ as a sub-complex of $S^{(n+1)}(V_y;G)$.
Identifying the points of $V_y$ and $f^{-1}(V_y)$ with the singular $0$-simplexes
in $V_y$ and $f^{-1}(V_y)$, respectively,
for every $z\in V_y$ we define $\varphi'(z)=x_z$ with $x_z\in f^{-1}(y)$,
and extend $\varphi'$ to a homomorphism
$\varphi:L\to S_0(X;G)$.
We can consider $\varphi$ as a correct chain morphism from $L$ into $S(X;G)$.
Since the identity morphism $\phi:S^{(n+1)}(V_y;G)\to S(Y;G)$ is correct and
$\mathcal V$-small and $f_\sharp\circ\varphi=\phi|L$,
according to Corollary 3.4,
$\varphi$ can be extended to a chain morphism
$\widetilde{\varphi}:S^{(n+1)}(V_y;G)\to S(X;G)$
such that
$\phi$ and $f_\sharp\circ\widetilde{\varphi}$ are $\mathcal U$-close.
Now, suppose $c^k$ is a singular $k$-cycle in $S^{(n+1)}(V_y;G)$ for some $k\leq n$.
Then
$\eta^k=\widetilde{\varphi}(c^k)$ is a cycle in $S_k(X;G)$ such that
$f_\sharp(\eta^k)$ is $\mathcal U$-close to $c^k$.
Since $f_\sharp\circ\varphi=\phi|L$,
the \lq\lq vertexes" of $\widetilde c^k=f_\sharp(\eta^k)$ and $c^k$ coincide,
so $|\widetilde c^k|\subset\rm{St}(V_y,\mathcal U)\subset W_y$.
Hence, $|\eta^k|\subset f^{-1}(W_y)$ and
there exists a $(k+1)$-chain $\eta^{k+1}\in S(f^{-1}(U_y);G)$ with $\partial\eta^{k+1}=\eta^{k}$.
This implies that
$\widetilde c^k$ is homologous to zero in $U_y$.
Therefore $Y$ is approximatively $lc^n_G$.
\end{proof}


\end{document}